\documentclass[11pt,a4paper]{article}

\usepackage[english]{babel}
\usepackage{graphicx} 
\usepackage{amsthm} 
\usepackage{amsmath}
\usepackage{amssymb}
\usepackage{mathtools}
\usepackage{amsthm} 
\usepackage{wasysym} 
\usepackage{mathrsfs} 
\usepackage{bbm} 
\usepackage{calligra} 
\usepackage{enumerate} 
\usepackage[shortlabels]{enumitem}
\usepackage{xcolor}
\usepackage{multirow} 
\usepackage{subfigure}
\usepackage{nicefrac}
\usepackage{comment}

\usepackage[hidelinks]{hyperref}

\newcommand{\w}{\omega}
\newcommand{\pa}{\partial}
\newcommand{\na}{\nabla}
\newcommand{\eps}{\varepsilon}
\newcommand{\N}{\mathbb{N}}
\newcommand{\Z}{\mathbb{Z}}

\newcommand{\R}{\mathbb{R}}

\newcommand{\LP}{\mathbb{P}}
\newcommand{\T}{\mathbb{T}}

\newcommand{\D}{\mathcal{D}}

\newcommand{\norm}[1]{\|#1\|}

\def\heat#1{e^{#1 \Delta}}

\newtheorem{thm}{Theorem}[section]
\newtheorem{prop}[thm]{Proposition}
\newtheorem{cor}[thm]{Corollary}
\newtheorem{lemma}[thm]{Lemma}
\newtheorem{defi}[thm]{Definition}
\newtheorem{rem}[thm]{Remark}

\title{Global well-posedness of the 3D Navier-Stokes\\ equations
with helical $L^1$ vorticity}
\author{Francisco Gancedo, Antonio Hidalgo-Torn\'e}

\begin{document}
	
\maketitle

\begin{abstract}
We show that the Navier-Stokes equations on $\mathbb R^2\times\mathbb T$
are globally well posed for initial vorticities that are both helically symmetric 
and integrable. This class of data typically generates velocity fields of infinite kinetic energy, so this result is not covered by the classical finite-energy well-posedness theory. This result is
supercritical from the perspective of the three-dimensional
Navier-Stokes scaling, and it is made possible by the special structure of
helical flows using
time-weighted Kato-type spaces. 
\end{abstract}
\section{Introduction}

We consider the incompressible Navier-Stokes equations
\begin{equation}\label{eq:NSintro}
\left\{
\begin{aligned}
    &\partial_t u + u\cdot\nabla u - \Delta u + \nabla P=0,
        \qquad t>0,\quad x\in\R^3,\\
    &\nabla\cdot u=0,\\
    &u|_{t=0}=u_0.
\end{aligned}
\right.
\end{equation}
We denote by
$
    \omega=\nabla\wedge u
$
the vorticity. Taking the curl of the first identity in \eqref{eq:NSintro}, one obtains
\begin{equation}\label{eq:vorticityintro}
    \partial_t\omega
    +u\cdot\nabla\omega
    -\omega\cdot\nabla u
    -\Delta\omega=0.
\end{equation}
Throughout the paper, the velocity $u$ associated with a divergence-free vorticity $\omega$ is always understood to be the velocity given by the Biot–Savart law
\[
u:=\nabla\wedge(-\Delta)^{-1}\omega.
\]
In particular, we do not add any curl-free harmonic component to the
velocity.

Global well-posedness for the three-dimensional Navier-Stokes equations remains a major open problem. In contrast, in two dimensions, Leray-Hopf solutions are globally well posed for divergence-free initial data in $L^2$; see for instance \cite{lemarierieusset2002,RobinsonRodrigoSadowski16,cheminfourier}.
From the point of view of the three-dimensional Navier-Stokes scaling,
vorticity in $L^{3/2}(\R^3)$ is critical. Indeed, if $(u,p)$ is a solution of
\eqref{eq:NSintro} on $\R^3$, then for every $\lambda>0$ the rescaled fields
$
    u_\lambda(t,x)=\lambda u(\lambda^2t,\lambda x),
    \quad
    p_\lambda(t,x)=\lambda^2 p(\lambda^2t,\lambda x),
$
also solve the Navier-Stokes equations, with vorticity
$
    \omega_\lambda(t,x)=\lambda^2\omega(\lambda^2t,\lambda x).
$
Consequently,
$
    \|\omega_\lambda(0)\|_{L^p(\mathbb R^3)}
    =
    \lambda^{2-\frac3p}
    \|\omega_0\|_{L^p(\mathbb R^3)}.
$
Thus $p=3/2$ is the critical exponent for vorticity. For $3/2<p<3$, the
Biot-Savart law places the associated velocity in a subcritical regularity
class, and the classical theory
\cite{kato1984} yields local well-posedness for arbitrary data in such
classes. At the critical level, one has global well-posedness under
smallness assumptions in suitable scale-invariant spaces; see, for
instance, \cite{kato1984,kozonoyamazaki1994,kochtataru2001,gigamiyakawa1989}. Even at the critical level, well-posedness may fail outside the small-data regime: non-uniqueness for smooth solutions emanating from large initial data in the critical space $BMO^{-1}$ was obtained in \cite{coiculescu2026}.
Below the critical threshold, the problem is supercritical, and one does
not expect a general well-posedness theory without additional structure.
In this direction, in \cite{BuckmasterVicol19}  
non-uniqueness of solutions with vorticity in $C([0,T];L^1(\T^3))$ is shown. See \cite{JMS21} for non-uniqueness with vorticity in 
$C([0,T];L^p(\T^3))$, $1<p<6/5$. 

In two dimensions, the $L^1$ norm of the vorticity,
and more generally the total variation of a finite measure, are invariant. Moreover, the vorticity equation is the scalar
transport-diffusion equation
$
    \partial_t\omega+u\cdot\nabla\omega-\Delta\omega=0,
    \,\, u=\nabla^\perp(-\Delta)^{-1}\omega,
$
and the vortex stretching term $\omega\cdot\nabla  u$ is absent. This allows one to solve the
Cauchy problem for rough initial vorticities. In particular, global
existence for finite measures as initial vorticity was obtained in
\cite{cottet1986,giga1988}, while uniqueness was proved in
\cite{gallaghergallaylions2005,gallaghergallay2005}; see also
\cite{gallaywayne2005,gallay2012,bedrossianmasmoudi2014}. There are also nonuniqueness results in two dimensions, just below the critical scale where well-posedness holds \cite{CheskidovLuo22,CheskidovLuo2023,BurczakHidalgoTorne2025}.

The purpose of this paper is to show that the situation in three dimensions changes
substantially under helical symmetry, which is a natural invariant structure for the Navier-Stokes equations. Helical symmetry is the invariance
under a simultaneous rotation around the vertical axis and translation
along it (see Definition \ref{def:helicsym}). A
natural domain without boundaries for such flows is $\mathbb R^2\times \gamma\mathbb T$, where $\gamma$ is a positive real number and $\T=\R/(2\pi\Z)$.

In
\cite{MahalovTitiLeibovich90}, global well-posedness in the
finite-energy setting is obtained. Further aspects of helically symmetric flows,
including limiting equations and vanishing-viscosity limits, were studied in
\cite{lopesmazzucatoniunussenzveig2014,jiulopesniunussenzveig2018}.
The finite-energy framework is natural in bounded domains. However, generic helical flows have infinite energy in the
unbounded setting $\mathbb R^2\times\mathbb \gamma \T$ because they have
two-dimensional behavior at spatial infinity.
A simple example is provided by the Oseen vortex, viewed as a helically
symmetric flow. Fix $\delta>0, \Gamma\neq 0$ and define
$$
    \omega^\delta_0(x)
    =
    \frac{\Gamma}{4\pi\delta}
    e^{-\frac{|x_h|^2}{4\delta}}\,e_3,
    \quad
    x=(x_h,x_3)\in\mathbb R^2\times \gamma \mathbb T.
$$
This vorticity is smooth, integrable, and
helically symmetric. The associated velocity is
$$
    u^\delta_0(x)
    =
    \frac{\Gamma}{2\pi |x_h|}
    \left(
        1-e^{-\frac{|x_h|^2}{4\delta}}
    \right)e_\theta,\quad\mbox{with}\quad e_\theta=(-\sin\theta,\cos\theta,0).
$$
As $|x_h|\to\infty$, one has
$$
    |u^\delta_0(x)|\sim \frac{|\Gamma|}{2\pi |x_h|}.
$$
Therefore,
$$
    \int_{\mathbb R^2\times \gamma \mathbb T}|u^\delta_0(x)|^2\,dx
    =
    \infty.
$$
Thus even very regular helically symmetric vorticities in $L^1$ may
generate velocity fields of infinite kinetic energy. This makes it natural
to work in local-energy or uniformly local spaces rather than in the usual
energy space.

In our previous work \cite{GH-T2023}, we proved global-in-time
existence for the Navier-Stokes equations starting from a helical
vortex filament. Short-time uniqueness was established in a
restrictive class requiring a local Oseen-type structure near the
filament as $t\to0^+$, as in \cite{bedrossiangermainharropgriffiths23}. At positive times, we obtained a broader
well-posedness theory for helically symmetric local energy weak
solutions with suitable spatial decay.
This showed that the helical structure can be used to obtain global existence beyond the finite-energy framework.
However, that approach did not provide global-in-time control of the
integrability of the vorticity in any Lebesgue space. In principle,
such estimates are obstructed by the combination of infinite-energy
velocities and vortex stretching. Also in the infinite-energy setting, \cite{Vila2024} proved the stability and asymptotics of the Lamb-Oseen vortex, viewed as a three-dimensional helical flow, under arbitrarily large $H^1$
 helical perturbations.

In the present paper, we consider helical and integrable initial vorticities
without further structural assumptions. Although the theory of \cite{GH-T2023} can be
applied in this setting, we avoid it for the reasons discussed above. 
The main result of the present paper shows that, under helical symmetry, $L^1$-initial vorticity is enough to obtain global well-posedness, with bounds on global spatial norms for any positive time. Such supercritical results are also possible under the assumption of axial symmetry, under the additional assumption of vanishing swirl \cite{GallaySverak2015}. Remarkably, in that setting it is also possible to establish the global existence of solutions when the initial vorticity is a measure \cite{FengSverak15}. More restrictively, uniqueness has been established for axisymmetric measure-valued vorticity restricted to the case of a single vortex ring \cite{GallaySverak19} or a finite combination of coaxial vortex rings with circulations of the same sign \cite{LevyLiu2018}.

The precise statement of our main Theorem is the following. To simplify the exposition, we fix the helical symmetry with $\gamma=1$. Consequently, we choose $\T=\R/(2\pi\Z)$. The Kato-type space $X_T^{4/3}$ (see definition \ref{KatoSpace}) is used to  capture the instantaneous gain of integrability.  

\begin{thm}[Global-in-time well-posedness]\label{thm:globalmain}
Let $\w_0\in L^1(\R^2\times\T)$ be helically symmetric and divergence free. Then, for any $T>0$, there exists a unique helically symmetric mild solution to \eqref{eq:NSDuhamel} in the class $\w \in C([0,\infty);L^1)\cap X^{4/3}_T$, with $\|\w\|_{X^{4/3}_t}\to 0$ as $t\to0^+ $. For any $t>0$, the vorticity $\w(t)$ belongs to the Sobolev spaces $W^{k,p}$ for any $k\in \N$ and $p\in [1,\infty]$.
\end{thm}

In other words, we prove that if
$
\omega_0\in L^1(\mathbb R^2\times\mathbb T)
$
is helically symmetric, then the Navier-Stokes equations are globally
well-posed in time-weighted Kato-type spaces $X^p_T$ adapted to the helical
heat-kernel scaling. 
The solution is unique in the corresponding Kato class. Moreover, we also quantify the blow-up rate in higher-order Sobolev norms as the solution approaches the initial data. See Theorem \ref{thm:higherregrate} below.

To obtain estimates for higher-order norms, we need to use the Kato–Ponce inequality, also known as the fractional Leibniz rule \cite{KatoPonce1vez}. It is known that this inequality holds on $\R^n$ \cite{GrafakosOh2014} and $\T^n$ \cite{BenyiOhZhao2025}, but, to the best of our knowledge, no proof for $\R^2\times \T$ is available in the literature. Building on the proof of this inequality in $\R^n$ \cite{GrafakosOh2014} and using a transference theorem \cite{RodriguezLopez2013}, we establish the corresponding inequality.

The paper is organized as follows: 
In Section \ref{sec:prelimanalysis}, we introduce the notions of helical symmetry and Kato-type spaces, establish some of their properties, and derive several preliminary results.
In Section \ref{sec:localexistence} we establish the local-in-time well-posedness of solutions. Section \ref{sec:globalintime} is devoted to the global-in-time unique continuation. Finally, in the Appendix we prove the Kato-Ponce inequality on domains of the form $\R^{n_1}\times \T^{n_2}$.

\subsection{Notation}
\begin{itemize}
\item $\Lambda=(-\Delta)^\frac{1}{2}$.
\item $\LP$ is the Leray Projector.
\item $L^p_TL^q_x=L^p((0,T);L^q(\R^2\times \T)).$ For time intervals not starting at $0$, we write $L^p_{[t_1,T]}L^q_x=L^p([t_1,T];L^q(\R^2\times \T)).$
\item We denote by $e^{t\Delta}$ the heat kernel in $\R^2\times \T$. To denote the heat kernel in a domain $D$, we write $e^{t\Delta}_D$.
\item For a tensor field $F=(F_{ij})$, we set
$$
(\nabla\cdot F)_j=\sum_i\partial_iF_{ij}.
$$

\item We use $f\lesssim g$ to denote $f\leq cg$, where $c$ is a constant that does not depend on any important parameter. We use $f\approx g$ to denote $f\lesssim g$ and $g\lesssim f$ at the same time.
\end{itemize}

\section{Preliminary analysis}\label{sec:prelimanalysis}

Since the following elementary lemma will be used several times throughout the paper, we include its statement and proof for completeness.
\begin{lemma}\label{lemma:intbeta}
If $-1<a,b\in \R$,
$$\int_0^t (t-\tau)^a \tau^b d\tau=t^{a+b+1}\int_0^1 (1-z)^a z^b dz=t^{a+b+1}\beta(a+1,b+1),$$
where $\beta(x,y)$ is the beta function.
\end{lemma}
\begin{proof}
Changing variables $\tau=tz$, we obtain the result. \end{proof}

\begin{cor}\label{cor:intbetatruncated}
If $a,b\in \R$ and $-1<a$, then 
$$\int_{\frac{t}{3}}^t (t-\tau)^a \tau^b d\tau=t^{a+b+1}\int_{\frac{1}{3}}^1 (1-z)^a z^b dz\lesssim  t^{a+b+1}.$$
\end{cor}

We now proceed to define helical symmetry. In the Cartesian coordinate system, consider the rotation matrix about the $x_3$-axis by an angle $\theta\in \mathbb{R}$
	\begin{equation*} \label{rotmat}
		\mathcal{R}_\theta :=
		\begin{bmatrix}
			\cos\theta & -\sin\theta  & 0\\[0.2cm]
			\sin\theta & \cos\theta   & 0\\[0.2cm]
			0          &    0         & 1
		\end{bmatrix},
	\end{equation*}
	and its combination with vertical translations
	$$
	\begin{bmatrix}
		x_1\\
		x_2\\
		x_3
	\end{bmatrix}
	\ \mapsto \mathcal{S}_{\theta,\gamma}\begin{bmatrix}
		x_1\\
		x_2\\
		x_3
	\end{bmatrix}
	:=\mathcal{R}_\theta\begin{bmatrix}
		x_1\\
		x_2\\
		x_3
	\end{bmatrix}+\gamma\begin{bmatrix}
		0\\
		0\\
		\theta
	\end{bmatrix},
	\quad\forall\theta\in\mathbb{R}\, ,\ \gamma \in\R\, .
	$$
\begin{defi}\label{def:helicsym}
We say that a smooth function $f:\mathbb{R}^3\to\mathbb{R}$ has helical symmetry if there exists $\gamma\neq0$ such that $f(x)=f(\mathcal{S}_{\theta,\gamma}x)$ for all $x\in\R^3$ and $\theta\in\mathbb{R}$. We say that a smooth vector field $u:\mathbb{R}^3\to\mathbb{R}^3$ has helical symmetry if there exists $\gamma\neq0$ such that $\mathcal{R}_\theta u(x)=u(\mathcal{S}_{\theta,\gamma}x)$ for all $x\in\R^3$ and $\theta\in\mathbb{R}$.

Denoting by $S_{\theta,\gamma}^{*}$ the pullback by $S_{\theta,\gamma}$, a scalar distribution
$f\in\mathcal{D}'(\mathbb{R}^3)$ is said to have helical symmetry if
$
S_{\theta,\gamma}^{*}f=f$ in $\mathcal{D}'(\mathbb{R}^3)
$
for every $\theta\in\mathbb{R}$. Similarly, a vector-valued distribution
$u\in\mathcal{D}'(\mathbb{R}^3;\mathbb{R}^3)$ is said to have
helical symmetry if
$
S_{\theta,\gamma}^{*}u=R_\theta u$ in $\mathcal{D}'(\mathbb{R}^3;\mathbb{R}^3)
$
for every $\theta\in\mathbb{R}$.
\end{defi}

\begin{cor}
A vector field $u$ has helical symmetry if and only if all its components in cylindrical coordinates are functions with helical symmetry. A smooth function $f$ has helical symmetry with parameter $\gamma\neq 0$ if and only if, when using cylindrical coordinates, it only depends on $\rho$ and $\xi=z-\gamma \theta$. 
\end{cor}
To simplify the exposition, from now on we fix $\gamma=1$ whenever we assume helical symmetry.

We state the following Poincaré-type inequality to relate the regularity of the velocity to that of the vorticity.   

\begin{lemma}\label{lem:helicalembedding}
Let $1<p<2$. For any $f$ with helical symmetry, $f\in L^q$, $\na f\in L^p$, and $\frac{1}{q}=\frac{1}{p}-\frac{1}{2}$, it holds
$$\|f\|_{L^q}\lesssim \|\na f\|_{L^p}.$$
\end{lemma}
\begin{proof}
We introduce cylindrical coordinates
$
(x_1,x_2,x_3)=(\rho\cos\theta,\rho\sin\theta,z)
$
and define the helical variable
$
\xi := z-\theta.
$
By helical symmetry, there exists a function $\tilde f(\rho ,\xi)$ such that
$
f(\rho,\theta,z)=\tilde f(\rho,z-\theta).
$
Since the change of variables $(\rho,\theta,z)\mapsto(\rho,\theta,\xi)$ has unit Jacobian,
$$
\begin{aligned}
\|f\|_{L^q(\mathbb{R}^2\times \T)}
&= \Big(\int_0^{2\pi}\int_0^{2\pi}\int_0^\infty
|\tilde f(\rho,\xi)|^q \, \rho d\rho d\xi d\theta\Big)^\frac{1}{q}
=\Big(2\pi\int_0^{2\pi}\int_0^\infty
|\tilde f(\rho,\xi)|^q  \rho d\rho d\xi\Big)^\frac{1}{q}\\
&\lesssim \Big(\int_0^{2\pi}\int_0^\infty
(|\pa_\rho \tilde f(\rho,\xi)|^p+|\frac{1}{\rho}\pa_\xi \tilde f(\rho,\xi)|^p) \rho d\rho d\xi\Big)^\frac{1}{p} \\
&\lesssim \Big(\int_0^{2\pi}\int_0^{2\pi}\int_0^\infty
(|\pa_\rho \tilde f(\rho,\xi)|^p+|\frac{1}{\rho}\pa_\xi \tilde f(\rho,\xi)|^p) \rho d\rho d\xi d\theta\Big)^\frac{1}{p},
\end{aligned}
$$
where we have applied the Sobolev embedding in two dimensions in polar coordinates. Undoing the change of variables and using that $\partial_\xi \approx \pa_\theta$ for helically symmetric functions, we conclude
$$
\|f\|_{L^q(\mathbb{R}^2\times \T)}
\lesssim \|\na f\|_{L^p(\mathbb{R}^2\times \T)}.$$
\end{proof}

\begin{rem}\label{uregularity} In particular, if $\omega$ has helical symmetry and $u=\nabla\wedge(-\Delta)^{-1}\omega$,
then, for $1<p<2$ and $\frac1q=\frac1p-\frac12$,
\[
\|u\|_{L^q}
 \lesssim \|\nabla u\|_{L^p}
 \lesssim \|\omega\|_{L^p},
\]
where the last inequality follows from the boundedness of the Riesz
transforms.
\end{rem}

The next lemma provides two inequalities used in the arguments for smoothing and long-time existence.

\begin{lemma}\label{lem:gagliardoniremberg}
Let $f$ be a smooth vector field with helical symmetry. Then,
$$\|f\|^2_{L^\infty}\lesssim \|f\|_{L^4}\|\na f\|_{L^4},$$
and
$$\|f\|^2_{L^4}\lesssim \|f\|_{L^2}\|\na f\|_{L^2}.$$
\end{lemma}
\begin{proof}
The proof is analogous to that of the previous lemma, with the Gagliardo–Nirenberg interpolation inequality replacing the Sobolev embeddings.
\end{proof}

The following lemma contains an inequality used in the local-existence and smoothing arguments.

\begin{lemma}\label{lemma:heathelicalestimates}
If $f$ has helical symmetry, $s\geq 0$, and $1\leq q\leq p< \infty$, then 
$$\|\Lambda^s e^{t\Delta}f\|_{L^p}\lesssim t^{\frac{1}{p}-\frac{1}{q}-\frac{s}{2}}\|f\|_{L^q}.$$
\end{lemma}
\begin{proof}
By \cite[Lemma A.6]{GH-T2023}, we have $\heat{t}(x)=\heat{t}_{\R^2}(x_1,x_2)\heat{t}_\T(x_3)$. Then, we apply Minkowski's Integral Inequality and heat kernel estimates in $\R^2$ to deduce
$$
\begin{aligned}
\|\heat{t}f\|_{L^p}=\|\|\heat{t}f\|_{L^p_{x_1,x_2}}\|_{L^p_{x_3}}\!\lesssim t^{\frac{1}{p}-\frac{1}{q}}\|\|{\heat{t}_\T}f\|_{L^q_{x_1,x_2}}\|_{L^p_{x_3}}
\!\lesssim t^{\frac{1}{p}-\frac{1}{q}}\|f\|_{L^q_{x_3}L^q_{x_1,x_2}}=t^{\frac{1}{p}-\frac{1}{q}}\|f\|_{L^q},
\end{aligned}
$$
where we have used that $\|f\|_{L^q_{x_1,x_2}}$ is independent of $x_3$ due to the symmetry, and therefore $\|f\|_{L^q_{x_1,x_2}}\approx\|f\|_{L^p_{x_3}L^q_{x_1,x_2}}\approx \|f\|_{L^q_{x_3}L^q_{x_1,x_2}}=\|f\|_{L^q}$. 

We observe that for any $s\in \N$, the kernel of $\na^s \heat{t}$ also preserves the separated-variable structure, and therefore the same argument can be repeated for any integer-order derivatives. The estimates for general $s>0$ follow from the moment inequality
for fractional powers of $-\Delta$
\cite[Proposition~6.6.4]{Haase2006}, applied to $e^{t\Delta}f$, which gives 
$$\norm{\Lambda^se^{t\Delta}f}_{L^p}\lesssim\norm{e^{t\Delta}f}_{L^p}^{1-\frac{s}{2m}}
\norm{\Delta^m e^{t\Delta}f}_{L^p}^{\frac{s}{2m}},$$
where $2m> s$ is an even integer. 
\end{proof}

The next definition introduces the Kato-type spaces $X_T^p$, which are Banach spaces that gain integrability.

\begin{defi}\label{KatoSpace}
For any $1< p\leq \infty$, we denote by $X^p_T$ the Kato-type spaces 
$$X^{p}_T=\{f\in \D'((0,T]\times \R^2\times\T):\,t^{1-\frac1p}f\in C((0,T];L^p)\},$$
such that the norm
$$\|f\|_{X^p_T}=\sup_{t\in(0,T]}t^{1-\frac1p}\|f(t)\|_{L^p}$$
is finite.
\end{defi}

These spaces have the following crucial properties for the convolution of the heat kernel with $L^1$ functions.

\begin{lemma}\label{lem:heatsmallXp}
Let $\w_0\in L^1$ with helical symmetry. Then, for any $1< p< \infty,$
$$\heat{t}\w_0\in X^p_T,\quad\mbox{and}\quad\|\heat{t}\w_0\|_{X^p_T}\to 0,\quad\mbox{as}\quad T\to 0^+.$$
\end{lemma}
\begin{proof}
For any $\eps>0$, there exists $\phi\in C^\infty_c$ helically symmetric and such that $\norm{\w_0-\phi}_{L^1}\leq \eps.$ Then,

$$\|\heat{t}\w_0\|_{X^p_T}\leq \|\heat{t}(\w_0-\phi)\|_{X^p_T}+\|\heat{t}\phi\|_{X^p_T}\lesssim \eps + T^{1-\frac{1}{p}}\|\phi\|_{L^p}, $$
where we have used Lemma \ref{lemma:heathelicalestimates}. Since $\eps$ is arbitrary, the proof is concluded.
\end{proof}

\begin{rem}\label{rem:heatsmallXphigherorder} By taking derivatives of the heat kernel we can show that, if $\w_0$ is an integrable function with helical symmetry,
\begin{equation}\label{eq:heatscaling}
t^{1-\frac{1}{p}+\frac{s}{2}}\|\Lambda^s\heat{t}\w_0\|_{L^p}\to 0 \quad \text{as}\quad t\to 0^+,   
\end{equation}
for any $s\geq 0$ and $1\leq p<\infty$, except for the case $(s,p)=(0,1).$
\end{rem}

We finally state the fixed-point result used to establish local-in-time existence.

\begin{prop}[Fixed-point Theorem]\label{prop:fixedpoint}
Let $X$ be a Banach space, $B$ a continuous bilinear map from $X\times X$ to $X$, and $\alpha$ a positive real number such that 
$$\alpha<\frac{1}{4\|B\|},\quad \text{with} \quad \|B\|\coloneqq \sup_{\|f\|,\|g\|\leq 1}\|B(f,g)\|.$$
Then, for any $f_0\in X$ with $\|f_0\|_X<\alpha$, there exists a unique $f$ with $\|f\|_X<2\alpha$ such that
$$f=f_0+B(f,f).$$
\end{prop}
\begin{proof}
A proof for this proposition can be found in \cite[Lemma 5.5]{cheminfourier}. 
\end{proof}

\section{Local-in-time existence}\label{sec:localexistence}
We use Duhamel's principle and the fact that $u,\w$ are divergence free to rewrite \eqref{eq:vorticityintro} as

\begin{equation}\label{eq:NSDuhamel}
\w(t)=\heat{t}\w_0+\int_0^t\na\heat{(t-\tau)}(\w\otimes u-u\otimes \w)(\tau)d\tau, \quad u=\na\wedge(-\Delta)^{-1}\w.
\end{equation}

\begin{thm}\label{thm:existencesmallp}
Let $\w_0\in L^1$ be helically symmetric and divergence free. Let $4/3\leq p<2$. Then, if $T\ll1$, there exists a unique mild solution $\w$ to the Navier-Stokes equations \eqref{eq:NSDuhamel} in $X^p_T$ with $\|\w\|_{X^p_T}\leq C(p)$, initial datum $\w_0$ and such that
$$\|\w\|_{X^p_t}\to 0 \quad \text{as}\quad t\to 0^+.$$
\end{thm}

\begin{proof}
Our goal is to apply Proposition \ref{prop:fixedpoint} to \eqref{eq:NSDuhamel} to obtain existence in $X^p_T$. Using the Biot-Savart law, we can see the nonlinear term in \eqref{eq:NSDuhamel} as a bilinear term in $\w$, i.e., 
$$\int_0^t\na\heat{(t-\tau)}(\w\otimes u-u\otimes \w)(\tau)d\tau\equiv B(\w,\w).$$
Note that the norms of $u=\nabla\wedge(-\Delta)^{-1}\omega$ and $\w=\na\wedge u$ can be related by means of Lemma \ref{lem:helicalembedding}. We now check that $B$ is a continuous operator on $X^p_T\times X^p_T \to X^p_T$. 

We take the norm $L^p$ and use lemma \ref{lemma:heathelicalestimates} with $\frac{1}{r}=\frac{1}{p}+\frac{1}{q}$ to get

\begin{align*}
\|B(\w,\w)(t)\|_{L^p}&\lesssim \int_0^t (t-\tau)^{\frac{1}{p}-\frac{1}{r}-\frac{1}{2}}\|u\otimes \w\|_{L^r}d\tau
\lesssim \int_0^t (t-\tau)^{\frac{1}{p}-\frac{1}{r}-\frac{1}{2}}\|u\|_{L^q} \|\w\|_{L^p}d\tau \\
&\lesssim_p \int_0^t (t-\tau)^{\frac{1}{p}-\frac{1}{r}-\frac{1}{2}} \|\w\|^2_{L^p}d\tau.
\end{align*}
In the last inequality above, we use Remark \ref{uregularity} to obtain
$
\frac{1}{q}=\frac{1}{p}-\frac{1}{2},
$
which imposes the condition $p<2$. This in turn gives
$
\frac{1}{r}=\frac{2}{p}-\frac{1}{2},
$
which imposes the second condition on $p$. Now we add the weights in time and use Lemma \ref{lemma:intbeta} to deduce 

$$\|B(\w,\w)(t)\|_{L^p}\lesssim_p \int_0^t (t-\tau)^{-\frac{1}{p}}\tau^{\frac{2}{p}-2}\|\w\|_{X^p_T}^2d\tau\lesssim_p t^{-1+\frac{1}{p}}\|\w\|_{X^p_T}^2,
$$
which implies, taking supremum in $t$,
$$\|B(\w,\w)\|_{X^p_T}\lesssim_p \|\w\|_{X^p_T}^2.$$

Using now Lemma \ref{lem:heatsmallXp}, we can apply Proposition \ref{prop:fixedpoint} to \eqref{eq:NSDuhamel}, concluding the proof.
\end{proof}

\begin{thm}\label{thm:shorttimeWP}
Let $\w_0\in L^1$ be helically symmetric and divergence free. Then, there exists $T>0, C>0$ such that there is a unique mild solution $\omega$ to the Navier-Stokes equations \eqref{eq:NSDuhamel} in $C([0,T];L^1)\cap X^{4/3}_T $, with $\|\w\|_{X^{4/3}_T}<C$ and initial datum $\w_0$. Besides, $\|\w\|_{X^{4/3}_t}\to 0$ as $t\to 0^+$.
\end{thm}
\begin{proof}
By Theorem \ref{thm:existencesmallp}, for $T\ll 1$ there exists $C>0$ and a unique solution with $\|\w\|_{X^{\frac{4}{3}}_T}<C.$ Taking $L^1$ norm in \eqref{eq:NSDuhamel}, we deduce  the following bound
\begin{align*}
\|\w(t)-\heat{t}\w_0\|_{L^1}\leq &\int_0^t\|\na\heat{(t-\tau)}(\w\otimes u-u\otimes \w)(\tau)\|_{L^1}d\tau\lesssim \int_0^t (t-\tau)^{-\frac{1}{2}}\|u\otimes \w\|_{L^1}d\tau \\ 
\leq& \int_0^t (t-\tau)^{-\frac{1}{2}}\|u\|_{L^4}\|\w\|_{L^\frac{4}{3}}d\tau
\lesssim \int_0^t (t-\tau)^{-\frac{1}{2}}\|\w\|^2_{L^\frac{4}{3}}d\tau
\lesssim \|\w\|_{X_t^\frac{4}{3}}^2.
\end{align*}
Theorem \ref{thm:existencesmallp} yields $\|\w\|_{X_t^\frac{4}{3}}\to 0$ as $t\to 0^+$, and therefore
$$\|\w(t)-\heat{t}\w_0\|_{L^1}\to 0\quad \text{as}\quad t\to 0^+.$$
Since $\|\heat{t}\w_0- \w_0\|_{L^1}\to 0$ as $t\to 0^+$, this gives that $\|\w(t)- \w_0\|_{L^1}\to 0$ as $t\to 0^+$. This shows the right continuity at 0 together with the time bound of the solution in $L^1$. 

Considering $t>0$ we study now the continuity at this point. Writing
$$
\w(s)=\heat{(s-t)}\w(t)+\int_0^{s-t}\na\heat{(s-t-\tau)}(\w\otimes u-u\otimes \w)(t+\tau)d\tau,
$$
the right continuity, $s\to t^+$, follows as before.
Considering $t>s>0$, we study now the left continuity, i.e., $s\to t^-$. We split
$$
\begin{aligned}
\w(t)-\w(s)=&(e^{t\Delta}-e^{s\Delta})\w_0+\int_0^s \big(e^{(t-s)\Delta}-I\big)\na\heat{(s-\tau)}(\w\otimes u-u\otimes \w)(\tau) d\tau \\
&+\int_s^t \na\heat{(t-\tau)}(\w\otimes u-u\otimes \w)(\tau) d\tau=J_1+J_2+J_3,
\end{aligned}$$
where $I$ denotes the identity operator. The term $J_1$ is continuous in $L^1$ by continuity of the heat kernel. The term $J_3$ is continuous in $L^1$, since
$$\int_s^t \|\na\heat{(t-\tau)}(\w\otimes u-u\otimes \w)(\tau)\|_{L^1} d\tau\lesssim \|\w\|_{X_T^\frac{4}{3}}^2\int_s^t(t-\tau)^{-\frac{1}{2}}\tau^{-\frac{1}{2}}d\tau\lesssim (t-s)^\frac{1}{2}s^{-\frac{1}{2}}.$$
Finally, for the term $J_2$, we apply the change $\sigma s=\tau$ to rewrite it as
$$J_2=\int_0^1 \big(e^{(t-s)\Delta}-I\big)\na\heat{s(1-\sigma)}(\w\otimes u-u\otimes \w)(s\sigma)s\,d\sigma.$$
The integrand above converges pointwise to $0$ as $s\to t^-$ and can be bounded in $L_x^1$ by
$$C(1-\sigma)^{-\frac{1}{2}}\sigma^{-\frac{1}{2}}\|\omega\|^2_{X_T^\frac{4}{3}}\in L^1(0,1),$$
for some constant $C>0$.
Hence, we can apply the Bochner dominated convergence theorem to conclude that $\|J_2\|_{L^1}\to 0$ as $s\to t^-.$
\end{proof}

\begin{rem}
Theorem \ref{thm:shorttimeWP} can be similarly proved for any $4/3\leq p<2$ instead of $4/3$. It is needed to fix one value (or a finite number of values) to obtain a positive $T$. The particular choice is not important as we will show in Theorem \ref{thm:higherregrate} that the solution belongs to any $X^p_T$.
\end{rem}
The following corollary will help us later to show that the solution obtained in Theorem \ref{thm:shorttimeWP} can be uniquely continued for all times.
\begin{cor}\label{cor:uniquenessNSLp}
Let $\w_1,\w_2$ be two helically symmetric mild solutions to \eqref{eq:NSDuhamel} with the same initial data, such that $\w_1,\w_2\in C([0,T];L^1\cap L^{\frac{4}{3}})$. Then, $\w_1=\w_2$ in $[0,T]$. 
\end{cor}
\begin{proof}
Let $\eta=\omega_1-\omega_2$ and
$v=u_1-u_2=\nabla\wedge(-\Delta)^{-1}\eta$. Set
$$
M:=\sup_{0\leq t\leq T}
\left(
\|\omega_1(t)\|_{L^{4/3}}
+\|\omega_2(t)\|_{L^{4/3}}
\right)<\infty.
$$
Since the two solutions have the same initial datum, the linear
term vanishes in the equation for $\eta$. Using Lemmas \ref{lem:helicalembedding} and \ref{lemma:heathelicalestimates},
$$
\|\eta(t)\|_{L^{4/3}}\lesssim
M\int_0^t(t-\tau)^{-3/4}\|\eta(\tau)\|_{L^{4/3}}\,d\tau.
$$
If $D(t):=\sup_{0\leq s\leq t}\|\eta(s)\|_{L^{4/3}},$
then
$$
D(t)\lesssim M t^{1/4}D(t).
$$
Consequently, choosing $T_0>0$ sufficiently small so that the
implicit constant times $M T_0^{1/4}$ is strictly smaller than
one, we conclude that $D(T_0)=0$. Thus
$\omega_1=\omega_2$ on $[0,T_0]$. Repeating the argument on successive
intervals of length $T_0$, we conclude that
$\omega_1=\omega_2$ on $[0,T]$.
\end{proof}

In the previous Corollary, the time continuity requirement for mild solutions can be replaced by an $L^\infty$ bound. More precisely, from
$$
\omega\in L^\infty([0,T];L^1\cap L^{4/3}),
$$
one can deduce
$$
\omega\in C([0,T];L^1\cap L^{4/3}).
$$
Indeed
$$
\|(\omega\otimes u-u\otimes\omega)(t)\|_{L^1}
 \lesssim \|\omega(t)\|_{L^{4/3}}\|u(t)\|_{L^4}
 \lesssim \|\omega(t)\|_{L^{4/3}}^2.
$$
Moreover,
$$
\|\nabla e^{(t-\tau)\Delta}(\omega\otimes u-u\otimes\omega)(\tau)\|_{L^1}
 \lesssim (t-\tau)^{-1/2}\|\omega(\tau)\|_{L^{4/3}}^2,
$$
and
$$
\|\nabla e^{(t-\tau)\Delta}(\omega\otimes u-u\otimes\omega)(\tau)\|_{L^{4/3}}
 \lesssim (t-\tau)^{-3/4}\|\omega(\tau)\|_{L^{4/3}}^2.
$$
Since both time singularities are integrable, the continuity follows
from the mild formulation, the strong continuity of the heat
semigroup, and a standard splitting of the Duhamel integral near
$t$.

\subsection{Further regularity and integrability}\label{sec:furtherregularity}

Taking into account the smoothing effect of the Navier-Stokes equations, and more specifically the Serrin condition for regularity \cite{RobinsonRodrigoSadowski16}, one can show that the solutions obtained in the previous section are smooth. However, this does not imply the integrability of higher-order norms.

We now establish decay rates for the higher integrability and differentiability norms of the solution constructed in Theorem \ref{thm:shorttimeWP}. We note that this approach is not as straightforward as incorporating higher-order norms into the fixed-point scheme, as doing so would violate the applicability conditions of Lemma \ref{lemma:intbeta}. Instead, we implement a bootstrapping argument based on \cite{dongli2009}.

\begin{thm}\label{thm:higherregrate}
Let $\w$ be the unique mild solution obtained in Theorem \ref{thm:shorttimeWP}, and $T$ its time of existence. Then, for any $s\geq 0$ and $1\leq p\leq  \infty$,
\begin{equation}\label{eq:scalingsolutionhigherregularity}
t^{1-\frac{1}{p}+\frac{s}{2}}\|\Lambda^s\w(t)\|_{L^p}\lesssim 1  
\end{equation}
for all $0<t\leq T.$ Besides, 
\begin{equation}\label{convergecetozero}
t^{1-\frac{1}{p}+\frac{s}{2}}\| \Lambda^s\w(t)\|_{L^p}\to 0 \text{ as }t\to 0, 
\end{equation}
except for the case $(s,p)=(0,1)$.
\end{thm}

\begin{proof}
The interpolation inequality
$$
t^{1-\frac{1}{p}}\|\omega(t)\|_{L^p}
\leq
\|\omega(t)\|_{L^1}^{\frac{4}{p}-3}
(t^{\frac14}\|\omega(t)\|_{L^{4/3}})^{4(1-\frac{1}{p})},
\qquad 1\leq p\leq \frac{4}{3},
$$
together with Theorem \ref{thm:shorttimeWP}, allows to extend the result to the case $s=0$ and $1<p<\frac{4}{3}$, thereby obtaining property \eqref{convergecetozero} as the same holds for $\norm{\w}_{X_t^{4/3}}$.

Next, we fix $4/3<p<2$, $0<t\leq T$, and $t/3<\sigma<2t/3,$ so that $\sigma\approx (t-\sigma)\approx t$. We take the $L^p$ norm in \eqref{eq:NSDuhamel}, and then we apply Lemma \ref{lemma:heathelicalestimates}, Theorem \ref{thm:shorttimeWP} and Remark \ref{uregularity} to deduce
\begin{equation}\label{eq:L2regularityestimate}
\begin{aligned}
\|\w(t)\|_{L^p}\leq& \|\heat{(t-\sigma)}\w(\sigma)\|_{L^p}+\int_\sigma^t\|\na\heat{(t-\tau)}(\w\otimes u-u\otimes \w)(\tau)\|_{L^p}d\tau\\
\lesssim & (t-\sigma)^{\frac{1}{p}-\frac{3}{4}}\|\omega(\sigma)\|_{L^{\frac{4}{3}}}+\int_\sigma^t (t-\tau)^{\frac{1}{p}-{1}-\frac{1}{2}}\|\w(\tau)\|_{L^\frac{4}{3}}\|u(\tau)\|_{L^4}d\tau \\
\lesssim & (t-\sigma)^{\frac{1}{p}-\frac{3}{4}}\sigma^{-\frac{1}{4}}\|\omega\|_{X_t^{\frac{4}{3}}}+\|\omega\|^2_{X_t^{\frac{4}{3}}}\int_\sigma^t (t-\tau)^{\frac{1}{p}-\frac{3}{2}}\tau^{-\frac{1}{4}}\tau^{-\frac{1}{4}}d\tau\\
\lesssim & t^{\frac{1}{p}-1}\|\omega\|_{X_t^{\frac{4}{3}}}+ t^{\frac{1}{p}-1}\|\omega\|^2_{X_t^{\frac{4}{3}}}\int_{\frac{1}{3}}^1 (1-z)^{\frac{1}{p}-\frac{3}{2}}z^{-\frac{1}{2}}dz.
\end{aligned}
\end{equation}
The inequality above provides the bound of $\norm{\w}_{X_t^p}$ for any fixed $4/3<p<2$, and this norm goes to 0 as $t\to0^+$.

We now extend the estimate to any $2<p<\infty$.
The linear part can be estimated as in \eqref{eq:L2regularityestimate}. For the nonlinear term, to ensure the boundedness of the time integral, we choose $q=\frac{2p}{p+1}$ when applying Lemma \ref{lemma:heathelicalestimates}. This gives
$$
\begin{aligned}
\|\w(t)\|_{L^p}\leq& \|\heat{(t-\sigma)}\w(\sigma)\|_{L^p}+\int_\sigma^t\|\na\heat{(t-\tau)}(\w\otimes u-u\otimes \w)(\tau)\|_{L^p}d\tau\\
\lesssim & t^{\frac{1}{p}-1}\|\omega\|_{X_t^{\frac{4}{3}}}+\int_\sigma^t (t-\tau)^{\frac{1}{p}-\frac{p+1}{2p}-\frac{1}{2}}\|\w(\tau)\|_{L^\frac{4p}{2p+1}}\|u(\tau)\|_{L^{4p}}d\tau\\
\lesssim & t^{\frac{1}{p}-1}\|\omega\|_{X_t^{\frac{4}{3}}}+ t^{\frac{1}{p}-1}\|\w\|^2_{X_t^\frac{4p}{2p+1}}\int_{\frac{1}{3}}^1 (1-z)^{\frac1{2p}-1}z^{\frac{1}{2p}-1}dz.
\end{aligned}$$
As before, we obtain the desired bound.
By interpolation, the result for $p=2$ is also obtained, concluding the proof for $s=0$ and $1\leq p<\infty.$

We now turn to higher-order derivatives. Lacking a priori estimates for derivative norms, we cannot directly include the target norm on the right-hand side of our estimates. Crucially, applying full derivatives to the nonlinear term yields a heat kernel exponent that violates the boundedness requirement from Corollary \ref{cor:intbetatruncated}. This singularity imposes the introduction of fractional derivatives.

Fix $1\leq p<\infty.$ We apply $\Lambda^\frac{1}{2}$ to \eqref{eq:NSDuhamel} and take $L^p$ norm. Proceeding as before, the linear term can be bounded using Remark \ref{rem:heatsmallXphigherorder}. For the nonlinear term, we can repeat the previous strategy as long as we choose the different norms properly. More specifically, we can take
$$
\begin{aligned}
\|\Lambda^\frac{1}{2}\w(t)\|_{L^p}\leq& \|\Lambda^\frac{1}{2}\heat{(t-\sigma)}\w(\sigma)\|_{L^p}+\int_\sigma^t\|\na\Lambda^{\frac{1}{2}}\heat{(t-\tau)}(\w\otimes u-u\otimes \w)(\tau)\|_{L^p}d\tau\\
\lesssim & t^{\frac{1}{p}-\frac{5}{4}}\|\omega\|_{X_t^{\frac{4}{3}}}+\int_\sigma^t (t-\tau)^{\frac{1}{p}-\frac{1}{p}-\frac{3}{4}}\|\w(\tau)\|_{L^{\frac{p(p+2)}{2}}}\|u(\tau)\|_{L^{p+2}}d\tau\\
\lesssim & t^{\frac{1}{p}-\frac{5}{4}}\|\omega\|_{X_t^{\frac{4}{3}}}+ t^{\frac{1}{p}-\frac{5}{4}}\|\omega\|_{X_t^{\frac{p(p+2)}{2}}}\|\omega\|_{X_t^{\frac{2p+4}{p+4}}}\int_{\frac{1}{3}}^1 (1-z)^{-\frac{3}{4}}z^{\frac{1}{p}-\frac{3}{2}}dz,
\end{aligned}$$
to obtain the desired result for $s=1/2$.

We now assume that the result is proved for any $1\leq p<\infty$ and $s=k/2$, for some $k\in \N$. We proceed by induction in $k$.  We apply $\Lambda^{\frac{k+1}{2}}$ to \eqref{eq:NSDuhamel}. To deal with the nonlinear term, we apply $k/2$ derivatives to $\omega\otimes u-u\otimes \w$, and $1/2$ to the heat kernel. We note that, as in Remark \ref{uregularity}, for any $s\geq 0$ and $1<p<2$, with $q$ defined by
$
\frac{1}{q}=\frac{1}{p}-\frac{1}{2},
$
we have
\begin{equation}\label{eq:bounduwgeneral}
\|\Lambda^s u\|_{L^q}\lesssim_q\|\Lambda^s\w\|_{L^p}.
\end{equation} Then, using Lemma \ref{lem:KatoPonce} we can conclude 

$$
\begin{aligned}
\|\Lambda^\frac{k+1}{2}\w(t)\|_{L^p}\leq& \|\Lambda^\frac{k+1}{2}\heat{(t-\sigma)}\w(\sigma)\|_{L^p}+\int_\sigma^t\|\na\Lambda^{\frac{1}{2}}\heat{(t-\tau)}\Lambda^\frac{k}{2}(\w\otimes u-u\otimes \w)(\tau)\|_{L^p}d\tau\\
\lesssim & t^{\frac{1}{p}-1-\frac{k+1}{4}}\|\omega\|_{X_t^{\frac{4}{3}}}+\int_\sigma^t (t-\tau)^{\frac{1}{p}-\frac{1}{p}-\frac{3}{4}}\big(\|\Lambda^\frac{k}{2}\w(\tau)\|_{L^{\frac{p(p+2)}{2}}}\|u(\tau)\|_{L^{p+2}}\\
&\hspace{5cm}+\|\w(\tau)\|_{L^{\frac{p(p+2)}{2}}}\|\Lambda^\frac{k}{2}u(\tau)\|_{L^{p+2}}\big)d\tau\\
\lesssim & t^{\frac{1}{p}-1-\frac{k+1}{4}}\|\omega\|_{X_t^{\frac{4}{3}}}+ t^{\frac{1}{p}-1-\frac{k+1}{4}}\int_{\frac{1}{3}}^1 (1-z)^{-\frac{3}{4}}z^{\frac{1}{p}-\frac{3}{2}-\frac{k}{4}}dz.
\end{aligned}$$
By interpolation, we obtain the result for any $s\in(0,\infty)$. 
For $p=\infty$, we apply Lemma \ref{lem:gagliardoniremberg} to
$f=\Lambda^s\omega$. Using also the boundedness of the Riesz
transforms in $L^4$, we obtain
$$
\begin{aligned}
t^{1+\frac s2}\|\Lambda^s\omega(t)\|_{L^\infty}
&\lesssim
\left(
t^{\frac34+\frac s2}
\|\Lambda^s\omega(t)\|_{L^4}
\right)^{\frac12}
\left(
t^{\frac54+\frac s2}
\|\Lambda^{s+1}\omega(t)\|_{L^4}
\right)^{\frac12}\lesssim 1.
\end{aligned}
$$
\end{proof}

\begin{rem}\label{rem:timecontinuity}
The mild formulation and the estimates above also imply that
$$
\Lambda^s\omega\in C((0,T];L^p),
\qquad s\geq0,\quad 1\leq p\leq \infty.
$$
Indeed, on every interval $[t_0,T]$ with $t_0>0$, one can restart
the mild formulation and argue as in
the discussion following Corollary \ref{cor:uniquenessNSLp}, after applying the corresponding
spatial derivatives. Consequently, except when $(s,p)=(0,1)$, the map
$$
t\longmapsto t^{1-\frac1p+\frac s2}\Lambda^s\omega(t)
$$
extends continuously to $[0,T]$ by assigning the value zero at
$t=0$. In the case $(s,p)=(0,1)$, the extension at $t=0$ is given by
$\omega_0$.
\end{rem}

\begin{rem}
Theorem \ref{thm:higherregrate} does not provide explicit control over the implicit constants. However, a more technical framework developed in \cite{dongli2009} achieves this explicit dependence, thereby enabling a proof that both velocity and vorticity remain analytic in space for every fixed $0 < t \leq T$. For simplicity, we restrict ourselves to proving this simpler version.
\end{rem}

\section{Global-in-time well-posedness}\label{sec:globalintime}
In this section, we extend the solution obtained in the previous section globally in time, while keeping the boundedness of all the norms considered. The strategy is to split the smooth solution obtained in the previous section in two parts. One part will be small, and the other will have finite energy. We will first solve Navier-Stokes for large times with the small part as initial data. Then, we will solve the remaining equation for the finite energy part by using energy estimates. Corollary \ref{cor:uniquenessNSLp} will guarantee that the decomposition is well-defined.

\begin{prop}\label{prop:longtimesmallpart}
Let $T>0$. Then, there exists $\eps>0$ such that for any divergence-free $\w_0$ with helical symmetry and satisfying
\begin{equation}\label{eq:smallnesscondition}
\|\w_0\|_{L^\frac{4}{3}\cap L^4}<\eps,
\end{equation}
there is a unique mild solution $\w$ to \eqref{eq:NSDuhamel} in $[0,T]$ with helical symmetry such that 
\begin{equation}\label{eq:smallnessproperty}
\|\w \|_{L^\infty_T(L_x^{\frac{4}{3}}\cap L_x^4)}<2\eps.
\end{equation}
\end{prop}
\begin{proof}
We will apply a fixed-point scheme in the space $L^\infty((0,T);L^\frac{4}{3}\cap L^4)$ using equation \eqref{eq:NSDuhamel}. We bound

$$
\begin{aligned}
\|\w(t)\|_{L^{\frac{4}{3}}}
\lesssim & \|\w_0\|_{L^\frac{4}{3}}+\int_0^t (t-\tau)^{\frac{3}{4}-1-\frac{1}{2}}\|\w(\tau)\|_{L^{\frac{4}{3}}}\|u(\tau)\|_{L^4}d\tau\\
\lesssim & \|\w_0\|_{L^\frac{4}{3}}+\|\w\|^2_{L^\infty_TL_x^{\frac{4}{3}}}\int_0^t (t-\tau)^{-\frac{3}{4}}d\tau\\
\lesssim & \|\w_0\|_{L^\frac{4}{3}}+\|\w\|^2_{L^\infty_TL_x^{\frac{4}{3}}} t^{\frac{1}{4}}\leq \|\w_0\|_{L^\frac{4}{3}}+\|\w\|^2_{L^\infty_TL_x^{\frac{4}{3}}} T^{\frac{1}{4}}.
\end{aligned}
$$
Analogously, 
$$
\begin{aligned}
\|\w(t)\|_{L^{4}}
\lesssim & \|\w_0\|_{L^4}+\int_0^t (t-\tau)^{\frac{1}{4}-\frac{1}{2}-\frac{1}{2}}\|\w(\tau)\|_{L^{4}}\|u(\tau)\|_{L^4}d\tau\\
\lesssim & \|\w_0\|_{L^4}+\|\w\|_{L^\infty_TL_x^4}\|\w\|_{L^\infty_TL_x^{\frac43}} t^{\frac{1}{4}}\leq \|\w_0\|_{L^4}+\|\w\|_{L^\infty_TL_x^4}\|\w\|_{L^\infty_TL_x^{\frac{4}{3}}}T^{\frac{1}{4}}.
\end{aligned}
$$
Thus, we can apply Proposition \ref{prop:fixedpoint} as long as $\|\w_0\|_{L^\frac{4}{3}\cap L^4}<CT^{-\frac{1}{4}}$, for some constant $C$.

\end{proof}
If we decompose $u=u_1+u_2$ and $u_1$ solves the Navier-Stokes equations \eqref{eq:NSintro}, then $u_2$ satisfies
\begin{equation}\label{eq:NSperturbedvelocity}
\begin{aligned}
&\pa_t u_2+\na\cdot (u_2\otimes u_2)+\na\cdot (u_1\otimes u_2)+\na\cdot (u_2\otimes u_1)-\Delta u_2=\nabla P_2,
\end{aligned}
\end{equation}
for some pressure $P_2$. The following Proposition ensures that \eqref{eq:NSperturbedvelocity} is well-posed in $H^1$.

\begin{prop}\label{prop:perturbedNSlongtimeWP}
Let $T>0$ and $\omega_1\in L^\infty_T(L^{4/3}_x\cap L^4_x)$ be divergence-free and helically symmetric, with associated Biot-Savart velocity $u_1=\nabla\wedge(-\Delta)^{-1}\omega_1$. 
Consider $u_2(0)\in H^1$ divergence-free and helically symmetric. Then, there exists a divergence-free and helically symmetric $u_2\in L^\infty_TH^1_x$ solving \eqref{eq:NSperturbedvelocity}. 
\end{prop}
\begin{proof}
Here we establish the a priori estimates. The regularization of the system and the passage to the limit follow from standard Navier-Stokes methods. 

We multiply \eqref{eq:NSperturbedvelocity} by $u_2$, integrate in space and use lemma \ref{lem:gagliardoniremberg} to deduce
$$
\begin{aligned}
\frac{1}{2}\pa_t\norm{u_2}^2_{L^2}+\norm{\na u_2}_{L^2}^2\lesssim  
\norm{\na u_1}_{L^2}\norm{u_2}_{L^4}^2
\lesssim \norm{\na u_1}_{L^2}\norm{u_2}_{L^2}\norm{\na u_2}_{L^2}.
\end{aligned}$$
We apply Young's inequality for products, the boundedness of the Riesz transforms on $L^2$, and interpolation to deduce
$$
\begin{aligned}
\pa_t\norm{u_2}^2_{L^2}+\norm{\na u_2}_{L^2}^2
\lesssim \norm{\na u_1}^2_{L^2}\norm{u_2}^2_{L^2}\lesssim\norm{\w_1}_{L^{\frac43}}\norm{\w_1}_{L^{4}}\norm{u_2}^2_{L^2}.
\end{aligned}$$
By Gronwall's inequality, 
\begin{equation}\label{LinfL2L2H1exp}
    \begin{aligned}
\norm{u_2(t)}^2_{L^2}+\int_0^t\norm{\na u_2(\tau)}_{L^2}^2d\tau\leq& \norm{u_2(0)}^2_{L^2} \exp(Ct ||\w _1||^2_{L^\infty_T(L_x^\frac{4}{3}\cap L_x^4)}),   
\end{aligned}
\end{equation}
for some constant $C>0$ and any $0\leq t\leq T$. 

Multiplying by $\Delta u_2$, we obtain similarly
$$
\begin{aligned}
\frac{1}{2}\pa_t\norm{\na u_2}^2_{L^2}\!+\!\norm{\Delta u_2}_{L^2}^2\lesssim  &
\norm{\Delta u_2}_{L^2}(\norm{u_2}_{L^4}\norm{\na u_2}_{L^4}\!+\!\norm{u_1}_{L^4}\norm{\na u_2}_{L^4}\!+\!\norm{u_2}_{L^4}\norm{\na u_1}_{L^4}).
\end{aligned}$$
Lemma \ref{lem:gagliardoniremberg} and the boundedness of the Riesz transform on $L^4$ yield
$$
\begin{aligned}
\frac{1}{2}\pa_t\norm{\na u_2}^2_{L^2}\!+\!\norm{\Delta u_2}_{L^2}^2\lesssim  &
\norm{\Delta u_2}_{L^2}^\frac{3}{2}\norm{u_2}^\frac{1}{2}_{L^2}\norm{\na u_2}_{L^2}\!+\!\norm{\Delta u_2}^\frac{3}{2}_{L^2}\norm{\w_1}_{L^\frac{4}{3}}\norm{\na u_2}^\frac{1}{2}_{L^2}\\
&\!+\!\norm{\Delta u_2}_{L^2}\norm{u_2}^\frac{1}{2}_{L^2}\norm{\na u_2}^\frac{1}{2}_{L^2}\norm{\w_1}_{L^4}.
\end{aligned}$$
By Young's inequality, we arrive at
$$
\begin{aligned}
\pa_t\norm{\na u_2}^2_{L^2}\!+\!\norm{\Delta u_2}_{L^2}^2\lesssim  &
\norm{u_2}^2_{L^2}\norm{\na u_2}^4_{L^2}\!+\!\norm{\w_1}^4_{L^\frac{4}{3}}\norm{\na u_2}^2_{L^2}\!+\!\norm{u_2}_{L^2}\norm{\na u_2}_{L^2}\norm{\w_1}^2_{L^4}.
\end{aligned}$$
Gronwall's inequality yields 
$$\begin{aligned}
\norm{\na u_2(t)}^2_{L^2}+\int_0^t\norm{\Delta u_2(\tau)}_{L^2}^2d\tau\leq \Big(&\norm{\na u_2(0)}^2_{L^2}
+\norm{\w_1}^4_{L^\infty_TL_x^\frac{4}{3}}\int_0^t\norm{\na u_2(\tau)}^2_{L^2}d\tau\\
&+\norm{\w_1}^2_{L^\infty_TL^4_x}\int_0^t \norm{u_2(\tau)}_{L^2}\norm{\na u_2(\tau)}_{L^2}d\tau\Big)\exp(Cg(t)),
\end{aligned}$$
with
$$g(t)=\int_0^t \norm{u_2(\tau)}_{L^2}^2\norm{\na u_2(\tau)}_{L^2}^2d\tau.$$
Using estimate \eqref{LinfL2L2H1exp}, the desired global-in-time control is obtained.
\end{proof}

\subsection{Proof of Theorem \ref{thm:globalmain}}
\begin{proof}
For short times, we can apply Theorem \ref{thm:shorttimeWP} and the results from Section \ref{sec:furtherregularity}. With this, we obtain a solution $\w$ to $\eqref{eq:NSDuhamel}$ in the time interval $[0,T_{loc}]$. We fix a time $0<t_1<T_{loc}$, and also an arbitrary $T>T_{loc}$. Let $\chi$ be a smooth function with cylindrical symmetry such that $\chi(r)\equiv 1$ for $r\leq 1$, and $\chi(r)\equiv 0$ for $r\geq 2$. For any $R>0,$ we define $\chi_R(r)=\chi(r/R)$.
We split the velocity $\nabla\wedge(-\Delta)^{-1}\w(t_1)=u(t_1)=u_1(t_1)+u_2(t_1)$, with

$$u_1(t_1)=\LP ((1-\chi_R) u(t_1)), \quad u_2(t_1)=\LP (\chi_R u(t_1)),$$
with $R\gg1$ so that $\nabla\wedge u_1(t_1)$ satisfies \eqref{eq:smallnesscondition}. This is possible because $\nabla\wedge u_1(t_1)=(1-\chi_R)\w(t_1)-\nabla\chi_R\wedge u(t_1)$, so the first term decays as $R\to \infty$ since $\w$ is in $L^p$ and the support of $(1-\chi_R)$ gets away from $0$, while the second one is supported in the annulus $A_R=\{R\leq |x_h|\leq 2R\}$ and therefore 
$$\norm{\nabla\chi_R\wedge u(t_1)}_{L^{\frac{4}{3}}}\lesssim \norm{u(t_1)}_{L^4(A_R)}\to 0, \quad \norm{\nabla\chi_R\wedge u(t_1)}_{L^{4}}\lesssim R^{-1}\norm{u(t_1)}_{L^4(A_R)}\to 0.$$

We then apply Propositions \ref{prop:longtimesmallpart} and \ref{prop:perturbedNSlongtimeWP} to $\nabla\wedge u_1(t_1)=\w_1(t_1)$ and $u_2(t_1)$, respectively. Writing $\nabla\wedge u_2=\omega_2$, we obtain a solution $\w=\omega_1+\omega_2$ in $[t_1,T]$ with $\w_1\in L^\infty([t_1,T];L^\frac{4}{3}\cap L^4)$ and $\w_2\in L^\infty([t_1,T];L^2)$. In particular, interpolation  provides $\w\in L^\infty([t_1,T];L^2)$. Using Duhamel's formula, we bound

$$
\begin{aligned}
\norm{\w(t)}_{L^{\frac{4}{3}}}
\lesssim & \norm{\w(t_1)}_{L^\frac{4}{3}}+\int_{t_1}^t (t-\tau)^{\frac{3}{4}-\frac{3}{4}-\frac{1}{2}}\norm{\w(\tau)}_{L^{2}}\norm{u(\tau)}_{L^4}d\tau\\
\lesssim & \norm{\w(t_1)}_{L^\frac{4}{3}}+\norm{\w}_{L^\infty_{[t_1,T]}L_x^{2}}(\norm{\w_1}_{L^\infty_{[t_1,T]}L^{\frac43}_x}+\norm{u_2}_{L^\infty_{[t_1,T]}H^1_x})T^\frac{1}{2},\\
\end{aligned}
$$
so $\w\in L^\infty([t_1,T];L^\frac{4}{3})$. Similarly, 
$$
\begin{aligned}
\norm{\w(t)}_{L^{1}}
\lesssim & \norm{\w(t_1)}_{L^1}+\int_{t_1}^t (t-\tau)^{1-1-\frac{1}{2}}\norm{\w(\tau)}_{L^{\frac{4}{3}}}\norm{u(\tau)}_{L^4}d\tau\\
\lesssim & \norm{\w(t_1)}_{L^1}+\norm{\w}_{L^\infty_{[t_1,T]}L_x^{\frac{4}{3}}}(\norm{\w_1}_{L^\infty_{[t_1,T]}L^{\frac43}_x}+\norm{u_2}_{L^\infty_{[t_1,T]}H^1_x})T^\frac{1}{2},\\
\end{aligned}
$$
so $\w\in L^\infty([t_1,T];L^1\cap L^2)$. 
Uniqueness for $t_1<t<T$ is granted by Corollary \ref{cor:uniquenessNSLp} and the observation following it. Since $T$ is arbitrary and $t_1$ is smaller than the local existence time $T_{loc}$, this shows global in time well-posedness. Finally, the uniform $L^{4/3}$ bound on compact time intervals allows
us to restart the local theory with a uniform positive lifespan.
Applying Theorem \ref{thm:higherregrate} and Remark \ref{rem:timecontinuity} to these restarted solutions and using uniqueness
to identify them with the global solution, we obtain
$$
\omega\in C([t,T];W^{k,p})
$$
for every $0<t<T<\infty$, $k\in\mathbb N$, and
$1\leq p\leq \infty$.
\end{proof}

\appendix

\section{Fractional Leibniz Rule in $\mathbb{R}^{n_1}\times\T^{n_2}$}
In this appendix, we prove the fractional Leibniz rule on domains of the form $\mathbb{R}^{n_1}\times\mathbb{T}^{n_2}$, which is used in Section \ref{sec:furtherregularity}.
\begin{lemma}\label{lem:KatoPonce}
Let $n_1,n_2\in \N\setminus{\{0\}}$, $\frac{1}{2}<p<\infty$, and
$$
1<q_1,q_2,r_1,r_2<\infty
$$
satisfying
$$
\frac{1}{p}
=
\frac{1}{q_1}+\frac{1}{r_1}
=
\frac{1}{q_2}+\frac{1}{r_2}.
$$
Given $s>\max(0, \frac{n_1+n_2}{p}-(n_1+n_2))$ or $s\in 2\mathbb{N}$, there exists a constant such that
for any $f,g\in \mathscr{S}(\R^{n_1}\times\T^{n_2})$ it holds
$$
\|\Lambda^s(fg)\|_{L^p}
\lesssim
\|\Lambda^s f\|_{L^{q_1}}\|g\|_{L^{r_1}}
+
\|f\|_{L^{q_2}}\|\Lambda^s g\|_{L^{r_2}}.
$$
\end{lemma}

\begin{proof}
We denote the frequency variable as
$$
\lambda=(\xi,k),\qquad
|\lambda|=(|\xi|^2+|k|^2)^{1/2},
$$ 
where $\xi\in\R^{n_1}$ and $k\in\Z^{n_2}.$ We now introduce the usual cutoffs from the Littlewood–Paley decomposition. Let $\Phi\in\mathscr{S}(\R^{n_1+n_2})$ be a radial function so that $\Phi(\lambda)\equiv 1$ for $|\lambda|\leq1$, and supported in the region $\{|\lambda|<2\}$. We define $\Psi(\lambda)=\Phi(\lambda)-\Phi(2\lambda)$.
In particular, we have
$$
\sum_{j\in\mathbb{Z}}\Psi(2^{-j}\lambda)=1,
\qquad \lambda\neq 0.
$$
We observe that the point $\lambda=0$ has Haar measure zero in
$\mathbb{R}^{n_1}\times\mathbb{Z}^{n_2}$.
Let $P_j$ denote the Fourier multiplier defined by
$$
\widehat{P_jf}(\lambda)
=
\Psi(2^{-j}\lambda)\widehat f(\lambda).
$$
The dyadic partition of unity gives
$$
\Lambda^s(fg) = \sum_{j,j'\in\mathbb Z}
\Lambda^s\bigl((P_jf)(P_{j'}g)\bigr).
$$
We split this sum into the regions $j'\leq j-2$,
$j\leq j'-2$, and $|j-j'|\leq1$, obtaining
$$
\begin{aligned}
\Lambda^s(fg)
={}&
\sum_{j\in\mathbb Z}
\Lambda^s\big(
(P_jf)\sum_{j'\leq j-2}P_{j'}g
\big)
+
\sum_{j'\in\mathbb Z}
\Lambda^s\big(
\big(\sum_{j\leq j'-2}P_jf\big)(P_{j'}g)
\big)
+
\sum_{\substack{j,j'\in\mathbb Z\\ |j-j'|\leq1}}
\Lambda^s\bigl((P_jf)(P_{j'}g)\bigr).
\end{aligned}
$$

Denote by $T_m$ the bilinear Fourier multiplier defined on $(\mathbb{R}^{n_1}\times\mathbb{T}^{n_2})^2$, with symbol $m$ defined on $(\mathbb{R}^{n_1}\times\mathbb{Z}^{n_2})^2$. Using that $\sum_{j'\leq j-2}\Psi(2^{-j'}\lambda)
=
\Phi(2^{-j+2}\lambda)$ for $\lambda\neq0$, we obtain
\begin{equation}
\label{eq:leibniz-decomposition}
\Lambda^s(fg)
=
T_{m_1}(\Lambda^s f,g)
+
T_{m_2}(f,\Lambda^s g)
+
T_{m_3}(f,\Lambda^s g),
\end{equation}
where
$$
\begin{aligned}
m_1(\lambda,\mu)
=&
\sum_{j\in\mathbb{Z}}
\Psi(2^{-j}\lambda)
\Phi(2^{-j+2}\mu)
\frac{|\lambda+\mu|^s}{|\lambda|^s},\\
m_2(\lambda,\mu)
=&
\sum_{j'\in\mathbb{Z}}
\Phi(2^{-j'+2}\lambda)
\Psi(2^{-j'}\mu)
\frac{|\lambda+\mu|^s}{|\mu|^s},\\
m_3(\lambda,\mu)
=&
\sum_{\substack{j,j'\in\mathbb{Z}\\ |j-j'|\leq 1}}
\Psi(2^{-j}\lambda)
\Psi(2^{-j'}\mu)
\frac{|\lambda+\mu|^s}{|\mu|^s}.
\end{aligned}
$$
We observe that the corresponding frequency
cutoff vanishes when the denominator vanishes, so the definitions above are well defined.

We now consider the corresponding symbols on $(\mathbb{R}^{n_1+n_2})^2$.
Let $M_i$, $i=1,2,3$, be defined by the same formulas as
$m_i$, but with
$$
\lambda,\mu\in\mathbb{R}^{n_1+n_2}.
$$
The symbols are related by
$$
m_i=M_i\circ(\pi\otimes\pi),
\qquad i=1,2,3,
$$
where
$$
\pi:\mathbb{R}^{n_1}\times\mathbb{Z}^{n_2}\longrightarrow\mathbb{R}^{n_1+n_2},
\qquad
\pi(\xi,k)=(\xi,k).
$$

The operators associated with $M_1,M_2,M_3$ share the structure of the three
operators $\Pi_1,\Pi_2,\Pi_3$ appearing in the proof of the
homogeneous $\R^n$-Kato-Ponce inequality in
\cite[Section~3, proof of Theorem~1]{GrafakosOh2014}.
In particular, their result gives for any $F,G\in\mathscr{S}(\R^{n_1+n_2})$ and $1/p=1/q_i+1/r_i$
\begin{equation}
\label{eq:euclidean-m23}
\|T_{M_i}^{\mathbb{R}^{n_1+n_2}}(F,G)\|_{L^p(\mathbb{R}^{n_1+n_2})}
\lesssim
\|F\|_{L^{q_i}(\mathbb{R}^{n_1+n_2})}
\|G\|_{L^{r_i}(\mathbb{R}^{n_1+n_2})},
\qquad i=1,2,3.
\end{equation}

We now apply the homomorphism theorem for bilinear multipliers \cite[Theorem~2.2 (i)]{RodriguezLopez2013} to the homomorphism $\pi$.
We can do it because the symbols $M_i$ are bounded and continuous on
\[
(\mathbb{R}^{n_1+n_2})^2\setminus\{(0,0)\}.
\]
Although there is no continuity at the joint origin, this does
not affect the argument. Indeed, by
\cite[Observation 3]{RodriguezLopez2013}, it is enough if the
continuity holds almost everywhere in $(\mathbb{R}^{n_1}\times\mathbb{Z}^{n_2})^2$.

Therefore,
\eqref{eq:euclidean-m23} transfers to
$(\mathbb{R}^{n_1}\times\mathbb{Z}^{n_2})^2$ and the symbols $m_i$.
Applying these estimates to
\eqref{eq:leibniz-decomposition} and choosing $(q_3,r_3)=(q_2,r_2)$, we conclude that
\[
\|\Lambda^s(fg)\|_{L^p}
\lesssim
\|\Lambda^s f\|_{L^{q_1}}\|g\|_{L^{r_1}}
+
\|f\|_{L^{q_2}}\|\Lambda^s g\|_{L^{r_2}},
\]
as desired.
\end{proof}

\begin{rem}
The same argument also yields the corresponding Kato-Ponce
inequality for the inhomogeneous fractional derivative
$J^s=(1-\Delta)^{s/2}$. Indeed, the required Euclidean estimates
are also proved in \cite{GrafakosOh2014}, and the corresponding
bilinear multipliers can be transferred to
$\mathbb{R}^{n_1}\times\mathbb{T}^{n_2}$ by the same application of
\cite[Theorem~2.2]{RodriguezLopez2013}.
\end{rem}
\section*{Acknowledgement}
This work originated during a visit by A.H.T. to F.G. at the Institute for Advanced Study in Princeton. F.G. and A.H.T. would like to thank the IAS for its hospitality and for providing an excellent research environment during the visit. F.G. and A.H.T. were partially supported by the AEI through the grants PID2022-140494NA-I00 and PID2024-158418NB-I00. F.G. was partially supported by the grant RED2022-134784-T (Spain) and the IMUS-Maria de Maeztu grant CEX2024-001517-M, funded by MICIU/AEI/10.13039/501100011033. A.H.T. was partially supported by the Institute for Advanced Study and the Max Planck Institute for Mathematics in the Sciences.

\bibliographystyle{abbrv}
\bibliography{tesis}

\begin{flushleft}
	Francisco Gancedo\\
	\textsc{Departamento de An\'alisis Matem\'atico \& Universidad de Sevilla.\\
	41012, Sevilla, Spain}\\
	\textit{E-mail address:
    fgancedo@us.es}
\end{flushleft}

\begin{flushleft}
	Antonio Hidalgo-Torn\'e\\
	\textsc{MPI für Mathematik in den Naturwissenschaften. \\
    04103, Leipzig, Germany.}\\
	\textit{E-mail address: hidalgo@mis.mpg.de}
\end{flushleft}

\end{document}